\documentclass[12pt,oneside,english]{amsart}

\usepackage[T1]{fontenc}
\usepackage[latin9]{inputenc}
\usepackage{geometry}
\usepackage{color}
\usepackage{babel}
\usepackage{amstext}
\usepackage{amsthm}
\usepackage{amssymb}
\usepackage[unicode=true,
 bookmarks=true,bookmarksnumbered=false,bookmarksopen=false,
 breaklinks=false,pdfborder={0 0 0},pdfborderstyle={},backref=false,colorlinks=true]
 {hyperref}
\usepackage{todonotes}

\makeatletter
\numberwithin{equation}{section}
\numberwithin{figure}{section}
\theoremstyle{plain}
\newtheorem{thm}{\protect\theoremname}[section]
\theoremstyle{plain}
\newtheorem{prop}{\protect\propositionname}[section]
\theoremstyle{plain}
\newtheorem{cor}[thm]{\protect\corollaryname}
\theoremstyle{plain}
\newtheorem{lem}[thm]{\protect\lemmaname}

\usepackage{bbm}

\newcommand{\N}{\mathbb{N}}

\allowdisplaybreaks

\newcommand{\1}{{\large\mathbb{1}}}
\usepackage{mathbbol}

\makeatother

\providecommand{\corollaryname}{Corollary}
\providecommand{\lemmaname}{Lemma}
\providecommand{\theoremname}{Theorem}
\providecommand{\propositionname}{Proposition}

\author{Sheldy Ombrosi}
\author{Israel P. Rivera-R\'{\i}os}
\address{Departamento de Matem\'atica, Universidad Nacional del Sur (UNS), Bah\'{\i}a
Blanca, Argentina and INMABB, Universidad Nacional del Sur (UNS)-CONICET,
Bah\'{\i}a Blanca, Argentina.}
\email{sombrosi@uns.edu.ar (Sheldy Ombrosi) }
\email{israel.rivera@uns.edu.ar (Israel P. Rivera-R\'{\i}os)}
\thanks{Both authors were partially supported by Agencia I+d+i PICT 2018-2501. The second author was partially supported by Agencia I+d+i PICT 2019-00018}

\begin{document}

\title{Weighted $L^{p}$ estimates on the infinite rooted $k$-ary tree}
\maketitle
\begin{abstract}In this paper, building upon ideas of Naor and Tao \cite{NT} and  continuing the study initiated in \cite{ORRS} by the authors and Safe, sufficient conditions are provided for weighted weak type and strong type $(p,p)$ estimates with $p>1$ for the centered maximal function on the infinite rooted $k$-ary tree to hold. Consequently a wider class of weights for those strong and weak type $(p,p)$ estimates than the one obtained in \cite{ORRS} is provided. Examples showing that the Sawyer type testing condition and the $A_p$ condition do not seem precise in this context are supplied as well. We also prove that strong  and weak type estimates are not equivalent, highlighting the pathological nature of the theory of weights in this setting.  Two weight counterparts of our conditions will be obtained as well.
\end{abstract}

\section{Introduction and Main results}
In the seventies Muckenhoupt \cite{M} established the following result.
If $1<p<\infty$ and $M$ stands for the Hardy-Littlewood the following
statements are equivalent 
\begin{itemize}
\item $w\in A_{p}$, namely 
\begin{equation}
[w]_{A_{p}}=\sup_{B}\frac{1}{|B|}\int_{B}w\left(\frac{1}{|B|}\int_{B}w^{-\frac{1}{p-1}}\right)^{p-1}<\infty\label{eq:ApM}
\end{equation}
where each $B$ is a ball of $\mathbb{R}^{n}$.
\item The strong type estimate
\begin{equation}
\|Mf\|_{L^{p}(w)}\lesssim\|f\|_{L^{p}(w)}\label{eq:SM}
\end{equation}
holds. 
\item The weak type estimate
\begin{equation}
\|Mf\|_{L^{p,\infty}(w)}\lesssim\|f\|_{L^{p}(w)}\label{eq:WM}
\end{equation}
holds. 
\end{itemize}
Throughout the next decades this seminal work lead to the development
of the theory of weights that has been a very fruitful area since
then. Also a number of generalizations for a number of settings has
been developed since them, for spaces of homogeneous type \cite{SWh,ABI,HK,K1,K2} ,
spaces of non-homogeneous type \cite{GCM,OP,TTV}, or even in the discrete setting \cite{BM,NW, HS,ST,STInf}. This note
can be framed in that last group. Let us provide our setting.

Given $k\geq2$ we denote by $T^{k}$ the infinite rooted $k$-ary
tree, namely, the infinite rooted tree such that each vertex has $k$
children. We shall write just $T$ in case there is no place to confusion.
Abusing of notation, we will also use $T$ to denote the set of the
vertices of the tree. We can define the metric measure space $(T,d,\mu)$
where $d$ is the usual tree metric, namely $d(x,y)$ is the number
of edges of the unique path between $x$ and $y$, and $\mu$ is the
counting measure defined on parts of the set of vertices. Given $A\subset T$
we denote $|A|=\mu(A)$ and 
\[
\int_{A}f(x)dx=\sum_{x\in A}f(x).
\]
We will also denote 
\begin{align*}
Mf(x) & =\sup_{r\in\mathbb{N}\cup\{0\}}\frac{1}{|B(x,r)|}\int_{B(x,r)}|f(y)|dy,\\
M^{\circ}f(x) & =\sup_{r\in\mathbb{N}\cup\{0\}}A_{r}^{\circ}f(x)
\end{align*}
where $A_{r}^{\circ}f(x)=\frac{1}{|S(x,r)|}\int_{S(x,r)}|f(y)|dy$
and
\begin{align*}
B(x,r) & =\{y\in T:d(x,y)\leq r\},\\
S(x,r) & =\{y\in T:d(x,y)=r\}.
\end{align*}
In contrast with the standard Euclidean setting, it makes sense to
consider this kind of maximal function since $S(x,r)$ are not sets
of measure $0$. Actually $|S(x,r)|\simeq k^{r}$. Furthermore, for
$k\geq2$, we have that $M^{\circ}\simeq M$ with constant independent
of $k$ (see \cite[Proposition 2.1]{ORRS}). 
At this point is important to note that the measure on $T$ is far from being doubling, even the upper doubling condition of \cite{HPub} fails. That fact disables the possibility of using the classical machinery available to deal the maximal function on $T$. However Naor and Tao \cite{NT} managed to overcome those difficulties using a combinatorial and expander argument to show that $M$ is of weak type $(1,1)$ with constant independent of $k$. Note that such a result
was obtained as well by Cowling, Meda and Setti in \cite{CMS} and that the strong type $(p,p)$
had already been obtained by Nevo and Stein in \cite{NS}.

Pushing forward techniques in \cite{NT}, the study of the theory
of weights for the infinite rooted $k$-ary tree was initiated in
\cite{ORRS}. The main results there were the following Fefferman-Stein
inequality
\[
w\left(\left\{ x\in T\,:\,Mf(x)>t\right\} \right)\lesssim\frac{1}{t}\int_{T}|f(x)|M(w^{s})(x)^{\frac{1}{s}}dx\qquad s>1
\]
and the quite surprising fact that there exists a weight $w$ such
that for any positive integer $n$, if $M^{n}w=(M\circ\stackrel{n\text{ times}}{\dots}\circ M)(w)$,
then $M^{n}w(x)\lesssim w(x)$ is not sufficient for the weighted
weak type $(1,1)$  nor even the weighted $(p,p)$ type for every $p>1$ to hold. This shows that the choice for
the maximal function in the right hand side of the Fefferman-Stein
inequality is sharp, in the sense that $M(w^{s})(x)^{\frac{1}{s}}$
cannot be replaced for any number of iterations of the maximal function. 

We would like to note as well that in \cite{ORRS} it was showed that there exist non trivial weights $w$ such that $M(w^s)(x)^\frac{1}{s}\lesssim w(x)$ and therefore there are weights such that the weighted weak type $(1, 1)$ estimate holds and hence the weighted strong type $(p, p)$ estimate for $p>1$ holds as well. A natural question is whether it is possible to obtain more general classes of weights for $p > 1$. In this work we will pursue that direction providing novel results in the range $p>1$. We will show that the class of suitable weights for this range is wider than the one obtained for the endpoint estimate in \cite{ORRS}.  

Studying this kind of weighted estimates may have as well connections with ergodic theory. 
Observe that when $k$ is odd $T$ is almost identifiable with the free group on $\frac{k+1}{2}$ generators. Naor and Tao settled settled the weak type $(1,1)$ estimate for the maximal function on the $k$-ary tree, but the non-amenability of the free group disabled the possibility of using standard arguments to transfer that result to that setting. Later on Tao \cite{T} disproved that result by showing that $f\in L^1$ is not enough for averaging operators to converge in that setting. That result of Tao implied the latter does not hold either since via convergence theorems in \cite{NS} it is possible to deduce the pointwise convergence from the weak type $(1,1)$ estimates.
On the other hand an earlier positive result due to Bufetov \cite{B} showed that strengthening conditions in Tao's result, namely, assuming that $f\in L\log L$ then those averaging operators converge. 

Coming back to weighted estimates, a first natural question would be studying the relationship between
some analogue of the $A_{p}$ condition as stated in (\ref{eq:ApM})
above, that in this case would read as $w\in A_{p}$ if
\begin{equation}\label{eq:ApCond}
\sup_{x\in T,r\geq0}\frac{1}{|B(x,r)|}\int_{B(x,r)}w\left(\frac{1}{|B(x,r)|}\int_{B(x,r)}w^{-\frac{1}{p-1}}\right)^{p-1}<\infty,
\end{equation}
and strong and weak type estimates. We shall prove that there exists
weights $w$ such that the weak type $(1,1)$ holds, and consequently
$M$ is bounded on $L^{p}(w)$ for every $p>1$, but also $w\not\in A_{p}$
for every $p>1$ (See Proposition \ref{Prop:ThmNeg1}). Consequently the
set of the balls $\{B(x,r)\}_{x\in T,r\geq0}$ is not a Muckenhoupt
basis as the ones considered, for instance, in \cite{DMRO}, since $A_{p}$ does not
characterize the strong type $(p,p)$ of the maximal function, in
contrast with the classical setting that we presented above in which
(\ref{eq:ApM}) and (\ref{eq:SM}) were equivalent conditions. Note
that our results differ from other results in the literature. For instance in Badr and Martell \cite{BM}, a doubling measure is considered, endowing the graph with a structure of space of homogeneous type and hence enabling them to use covering lemmas. In the case of Hebisch, Steger \cite{HS}, the setting considered allow them to have a Calder\'on-Zygmund type decomposition, which is not available in our case either.

Another pathological situation in our setting is that there exist
weights such that (\ref{eq:WM}) holds but (\ref{eq:SM}) does not
(See Proposition \ref{Prop:Neg2}). Due to the aforementioned results it seems clear that $A_{p}$ type conditions do not seem neither precise nor suitable in
this setting. Observe that if we assume that the strong type $(p,p)$ with respect to a weight $w$
holds and $E,F\subset T$ then we have that 
\begin{equation}
\begin{split} & \1\otimes w\left(\{(x,y)\in E\times F:\ d(x,y)=r\}\right)\simeq k^{r}\int_{E}A_{r}^{\circ}(\chi_{F}w)=k^{r}\int_{F}A_{r}^{\circ}(\chi_{E})w\\
 &\leq k^{r}w(F)^{1-\frac{1}{p}}\left(\int_{T}A_{r}^{\circ}(\chi_{E})^pw\right)^{\frac{1}{p}} \lesssim k^{r}w(F)^{1-\frac{1}{p}}w(E)^{\frac{1}{p}}
\end{split}
\label{eq:NecCond}
\end{equation}
where $\1\otimes w$ stands for the product measure, namely
\[
\1\otimes w(A\times B)=|A|w(B).
\]
It is not difficult to see that in general the necessary condition (\ref{eq:NecCond}) will be not a sufficient condition, nor  even for the weighted weak type $(p,p)$ estimate. However assuming a slightly better decay on $k$ we manage to obtain a sufficient condition which is one of the main results of this paper. 

\begin{thm}
\label{thm:Suff}Let $k\geq2$ be an integer. Let $p>1$ and let $w$
be a weight. Assume that there exist $0<\beta<1$ and $\beta\leq\alpha<p$
such that 
\begin{equation}
\1\otimes w\left(\{(x,y)\in E\times F:\ d(x,y)=r\}\right)\lesssim k^{r\beta}w(E)^{\frac{\alpha}{p}}w(F)^{1-\frac{\alpha}{p}}.\label{eq:SuffCond}
\end{equation}
Then 
\begin{equation}
\|Mf\|_{L^{\frac{\beta}{\alpha}p,\infty}(w)}\lesssim\|f\|_{L^{\frac{\beta}{\alpha}p}(w)}.\label{eq:Weakbpalpha}
\end{equation}
Furthermore, if $\beta<\alpha$ then 
\begin{equation}
\|Mf\|_{L^{p}(w)}\lesssim\|f\|_{L^{p}(w)}\label{eq:Strongpp}
\end{equation}
and also 
\begin{equation}
\|Mf\|_{L^{p'}(\sigma_{p})}\lesssim\|f\|_{L^{p'}(\sigma_{p})}\label{eq:StrongDual}
\end{equation}
where $\sigma_{p}=w^{-\frac{1}{p-1}}$.
\end{thm}

At this point some remarks are in order. First observe that that if $0<\frac{\beta}{\alpha}<1$ (\ref{eq:Strongpp})
follows from (\ref{eq:Weakbpalpha}) by interpolation, however that
is not the case for (\ref{eq:StrongDual}). We will settle both estimates
relying upon the fact that under that assumption on $\beta$ and $\alpha$,
actually we can show something slightly stronger, namely that 
\[
\sum_{r=0}^{\infty}\|A_{r}^{\circ}f\|_{L^{p}(w)}\lesssim\|f\|_{L^{p}(w)}.
\]
Our condition is also sharp in the sense that if $\beta=\alpha$ we can provide a weight for which the weak type $(p,p)$ holds but the strong type $(p,p)$ does not and consequently the weak type $(q,q)$ fails as well for every $q<p$ (see Proposition \ref{Prop:Neg2}). 

Observe also that in the classical setting $w\in A_p$ implies that $\sigma\in A_{p'}$ and hence no additional argument is required to show that \eqref{eq:StrongDual} holds. Hence, even though we do not necessarily have that the fact that $w$ satisfies \eqref{eq:SuffCond} implies that $\sigma$ satisfies an analogue condition, Theorem \ref{thm:Suff} allows to derive from \eqref{eq:SuffCond} as well to settle \eqref{eq:StrongDual}. This, however may not be the case for the weak type estimates (see Proposition \ref{Prop:Neg2}).

Checking the sufficient condition in Theorem \ref{thm:Suff} directly could be quite involved. The following Corollary reduces the question to a condition that is easier to check than \eqref{eq:SuffCond}.
Before presenting that result it is convenient to introduce a bit of notation. In what follows we shall call $T_{0}$ the set given by the root of the tree, $T_{j}$ the set of the children of the vertices in $T_{j-1}$ and so on.
Observe that the sets considered in the condition in \eqref{eq:SuffCond} may have non-empty intersection with several different levels. With the condition in the following Corollary it suffices to study the behavior of the weight in each level instead.
\begin{cor}
\label{cor:SuffAux}Let $1< p<\infty$, and let $w$ be a weight such that there exists a real number $\delta<1$
such that if $x\in T_{j}$ and $|i-j|\leq r$ we have that
\begin{equation}
w(T_i\cap S(x,r))\lesssim k^{\frac{r+i-j}{2}(p-\delta)}k^{r\delta}w(x).\label{eq:Suffweaker}
\end{equation}
Then 
\begin{equation*}
\|Mf\|_{L^{p,\infty}(w)}\lesssim\|f\|_{L^{p}(w)}.
\end{equation*}
Furthermore, if $p<q$ then 
\begin{equation*}
\|Mf\|_{L^{q}(w)}\lesssim\|f\|_{L^{q}(w)}
\end{equation*}
and also 
\begin{equation*}
\|Mf\|_{L^{q'}(\sigma_{q})}\lesssim\|f\|_{L^{q'}(\sigma_{q})}
\end{equation*}
where $\sigma_{q}=w^{-\frac{1}{q-1}}$.
\end{cor}
It is worth observing that if for a certain $p>1$,  \eqref{eq:Suffweaker} holds, then it holds as well for every $q>p$. This fact will be crucial for us to settle the strong type estimates in Corollary \ref{cor:SuffAux}.

We would like observe that as in the case of the condition in Theorem \ref{thm:Suff}, the example in Proposition \ref{Prop:Neg2} also shows that \eqref{eq:Suffweaker}  in Corollary \ref{cor:SuffAux} is sharp.

From Theorem \ref{thm:Suff} and Corollary \ref{cor:SuffAux} we shall derive some Corollaries. In Corollary \ref{cor:kalpha} will provide weights for which the endpoint estimate does not hold but weak and strong type estimates do. In Corollary \ref{cor:A1Ap} we will show that the condition obtained in \cite{ORRS} also allows to provide strong type estimates. Observe that in that case even though that the strong type $(p,p)$ estimate follows by interpolation, that is not the case for the strong $(p',p')$ estimate in terms of the dual weight, and then Theorem \ref{thm:Suff} yields new estimates. We wonder whether if, as in the classical setting, (\ref{eq:Strongpp}) is equivalent to  (\ref{eq:StrongDual})  in this setting.

The remainder of the paper is organized as follows. 
Section \ref{sec:Examples} is devoted to provide the examples of weights announced in the introduction. In Section \ref{sec:Pathological}
we give precise statements and proofs relative to the $A_{p}$ condition in this setting. In Section \ref{sec:SuffCorolariesproof} we give detailed proofs of  Theorem
\ref{thm:Suff} and Corollary \ref{cor:SuffAux}. Finally, in Section \ref{sec:twoWeights} we provide our two weight results and a remark concerning the Sawyer testing condition.

\section{Examples}\label{sec:Examples}
This section is devoted to present the examples of weights that we announced in the introduction. The following Corollary shows that there exist weights beyond the class studied in \cite{ORRS} for which the weighted strong and weak type estimates hold.
\begin{cor}
\label{cor:kalpha}Let $p>1$. If
\[
w(x)=\sum_{j=0}^{\infty}k^{j(p-1)}\chi_{T_{j}}(x)
\]
we have that
\[
\|Mf\|_{L^{p,\infty}(w)}\lesssim\|f\|_{L^{p}(w)}
\]
and also that for every $q>p$
\[
\|Mf\|_{L^{q}(w)}\lesssim\|f\|_{L^{q}(w)}
\]
and 
\[
\|Mf\|_{L^{q}\left(\sigma_{q}\right)}\lesssim\|f\|_{L^{q}\left(\sigma_{q}\right)}.
\]
\end{cor}
\begin{proof}
It suffices to show that \eqref{eq:Suffweaker} holds. 
Let $x\in T_{j}$ and assume that $|i-j|\leq r$
and that $m\in\{0,\dots,r\}$ is the only integer such that $i=j+r-2m$.
Then
\[
\begin{split} & w(T_{i}\cap S(x,r))\\
 & =|T_{i}\cap S(x,r)|k^{i(p-1)}\leq k^{i(p-1)}k^{r-m}\\
 & =k^{(i-j)(p-1)}k^{r-m}k^{j(p-1)}=k^{(r-2m)(p-1)}k^{r-m}w(x)\\
 & =k^{(r-m)(2p-1)}k^{-p(r-m)}k^{-m(p-1)}k^{r-m}w(x)\\
 & =k^{(r-m)(2p-1)}k^{(1-p)r}w(x)
\end{split}
\]
Hence  \eqref{eq:Suffweaker} holds with $\delta=1-p$, since $\frac{i-j+r}{2}=r-m$.
\end{proof}
Our next corollary recovers the weighted estimates that could be derived from \cite{ORRS} via interpolation and furthermore we obtain the strong type estimates for the dual weight.

\begin{cor}
\label{cor:A1Ap}If there exists some $s>1$ such that 
\[
M_{s}w\lesssim w
\]
then we have that for every $p>1$
\[
\|Mf\|_{L^{p}(w)}\lesssim\|f\|_{L^p(w)}
\]
and also that 
\[
\|Mf\|_{L^{p'}(\sigma_{p})}\lesssim\|f\|_{L^{p'}(\sigma_{p})}.
\]
\end{cor}
\begin{proof}
Besides the interpolation from the endpoint we may provide an alternative
argument relying upon Theorem \ref{thm:Suff}. Recall that in \cite{ORRS}
it was shown that for $1<s<\infty$
\[
\1\otimes w\left(\{(x,y)\in E\times F:\ d(x,y)=r\}\right)\lesssim k^{r\frac{s'}{s'+1}}w(E)^{\frac{s'}{s'+1}}w(F)^{\frac{1}{s'+1}}.
\]
Observe that then
\[
\1\otimes w\left(\{(x,y)\in E\times F:\ d(x,y)=r\}\right)\lesssim k^{r\frac{s'}{s'+1}}w(E)^{\frac{\frac{s'}{s'+1}p}{p}}w(F)^{1-\frac{\frac{s'}{s'+1}p}{p}}
\]
and consequently the result follows from Theorem \ref{thm:Suff} with 
\[
\beta=\frac{s'}{s'+1}\qquad\alpha=\frac{s'}{s'+1}p.
\]
\end{proof}

\section{\label{sec:Pathological}Differences with the classical theory}

As we mentioned in the introduction a natural question is the relationship
of the weighted boundedness of the maximal function on $k$-ary trees
with constant independent of $k$ and the $A_{p}$ condition presented in \eqref{eq:ApCond}.
 We are going to provide two results that are slightly stronger, since we are going to settle them with balls $B(x,r)$ replaced by spheres $S(x,r)$ which yields a more restrictive condition since for $h\geq0$
 \[\frac{1}{|S(x,r)|}\int_{S(x,r)}h\lesssim \frac{1}{|B(x,r)|}\int_{B(x,r)}h \]
 as it was shown in \cite[Proposition 2.1]{ORRS}.

Our first result shows that we can find weights for which the maximal
function is bounded for every $q>1$ but $w\not\in A_{p}$ for every
$p>$1.
\begin{prop}
\label{Prop:ThmNeg1}Let $\delta\in\left(\frac{1}{2},1\right)$ 
\[
w(x)=\sum_{j=0}^{\infty}\frac{1}{k^{\delta j}}\chi_{T_{j}}(x)
\]
then, for every $q>1$
\[
\|M\|_{L^{q}(w)}\lesssim\|f\|_{L^{q}(w)}
\]
but also for every $p>1$
\[
\sup_{x\in T,r>0}\frac{1}{|S(x,r)|}\int_{S(x,r)}w\left(\frac{1}{|S(x,r)|}\int_{S(x,r)}w^{-\frac{1}{p-1}}\right)^{p-1}=\infty.
\]
\end{prop}
\begin{proof}
Note that the boundedness follows from Corollary \ref{cor:A1Ap}. Now we shall
show that $w$ does not satisfy the classical $A_{p}$ condition for
any $p>1$. Let us fix $p>1$. Let $x\in T_{j}$ and let us consider
$S(x,j)$. Then 
\[
\frac{1}{|S(x,j)|}\int_{S(x,j)}w\geq\frac{1}{|S(x,j)|}\int_{S(x,j)}w\gtrsim\frac{1}{k^{j}}.
\]
On the other hand 
\begin{align*}
\left(\frac{1}{|S(x,j)|}\int_{S(x,j)}w^{-\frac{1}{p-1}}\right)^{p-1} & \gtrsim\left(\frac{1}{k^{j}}\int_{B\cap T_{2j}}w^{-\frac{1}{p-1}}\right)^{p-1}\\
 & \geq\left(\frac{k^{2j\delta\frac{1}{p-1}}}{k^{j}}k^{j}\right)^{p-1}=k^{2j\delta}.
\end{align*}
Hence
\[
\frac{1}{|S(x,j)|}\int_{S(x,j)}w\frac{1}{|S(x,j)|}\int_{S(x,j)}w^{-1}\gtrsim k^{2j\delta}\frac{1}{k^{j}}=k^{j(2\delta-1)}
\]
and letting $j\rightarrow\infty$ the desired conclusion follows.
\end{proof}
Another fundamental difference between the classical theory and the
theory on $k$-ary trees is that there exist weights such that the
weighted weak type inequality holds but that is not the case for the
strong type.
\begin{prop}
\label{Prop:Neg2}Let $p>1$ and
\[
w(x)=\sum_{j=0}^{\infty}k^{(p-1)j}\chi_{T_{j}}(x).
\]
Then $\|M\|_{L^{p,\infty}(w)}\lesssim\|f\|_{L^{p}(w)}$ but 
\[
\|Mf\|_{L^{p}(w)}\lesssim\|f\|_{L^{p}(w)}
\]
does not hold.
\end{prop}
\begin{proof}
Note that the weak type $(p,p)$ estimate with respect to $w$ follows
from Corollary \ref{cor:kalpha}. Let $f=\chi_{T_{j}}$. Then we have that
\[
\|f\|_{L^{p}(w)}^{p}=\int_{T}|\chi_{T_{j}}|^{p}w=\int_{T_{j}}k^{(p-1)j}=k^{pj}.
\]
On the other hand 
\begin{align*}
\|Mf\|_{L^{p}(w)}^{p} & =\int_{T}(Mf)^{p}w\geq\sum_{i>j}\int_{T_{i}}(Mf)^{p}w\\
 & \geq\sum_{i>j}\int_{T_{i}}\frac{1}{k^{p(i-j)}}k^{i(p-1)}=\sum_{i>j}\frac{1}{k^{pi-pj}}k^{pi}\\
 & =k^{pj}\sum_{i>j}1
\end{align*}
and combining the estimates above
\[
k^{pj}\sum_{i>j}1\leq\|Mf\|_{L^{p}(w)}^{p}\lesssim\|f\|_{L^{2}(w)}^{p}=k^{pj}
\]
from what follows that 
\[
k^{pj}\sum_{i>j}\frac{1}{k^{i}}\lesssim k^{pj}\iff\sum_{i>j}1\lesssim1
\]
which is a contradiction.
\end{proof}

\section{Proof of Theorem  \ref{thm:Suff} and Corollary \ref{cor:SuffAux}}\label{sec:SuffCorolariesproof}

\subsection{Proof of Theorem \ref{thm:Suff}}

 We will follow the scheme devised in Naor and Tao \cite{NT} and further exploited in \cite{ORRS}.
We begin with the following Lemma.
\begin{lem}
\label{lem:sumLevels} Let $k\geq2$ be an integer. Let $p>1$ and
let $w$ be a weight. Assume that there exist $0<\beta<1$ and $\beta\leq\alpha<p$
such that 
\[
\1\otimes w\left(\{(x,y)\in E\times F:\ d(x,y)=r\}\right)\lesssim k^{r\beta}w(E)^{\frac{\alpha}{p}}w(F)^{1-\frac{\alpha}{p}}.
\]
Let $r>0$ and $\lambda>0$. Then 
\[
w\left(\left\{ A_{r}^{\circ}(|f|)\geq\lambda\right\} \right)\lesssim\sum_{\substack{n\in\N\cup\{0\}\\
1\leq2^{n}\leq2k^{r}
}
}\left(\frac{2^{n}}{k^{r}}\right)^{\frac{1-\beta}{2}\frac{p}{\alpha}}2^{\beta\frac{p}{\alpha}n}w\left(\left\{ |f|\geq2^{n-1}\lambda\right\} \right).
\]
\end{lem}

\begin{proof}
We can assume without loss of generality $f$ to be non-negative.
We bound 
\begin{equation}
|f|\leq\frac{1}{2}+\sum_{\substack{n\in\N\cup\{0\}\\
1\leq2^{n}\leq k^{r}
}
}2^{n}\chi_{E_{n}}+|f|\chi_{\{|f|\geq\frac{1}{2}k^{r}\}},\label{eq:troceado}
\end{equation}
where $E_{n}$ is the sublevel set 
\begin{equation}
E_{n}=\left\{ 2^{n-1}\leq|f|<2^{n}\right\} .\label{eq:niveles}
\end{equation}
Hence 
\begin{equation}
A_{r}^{\circ}(|f|)\leq\frac{1}{2}+\sum_{\substack{n\in\N\cup\{0\}\\
1\leq2^{n}\leq k^{r}
}
}2^{n}A_{r}^{\circ}\left(\chi_{E_{n}}\right)+A_{r}^{\circ}\left(f\chi_{\{f\geq\frac{1}{2}k^{r}\}}\right).\label{eq:troceadopromediado}
\end{equation}
First we observe that
\begin{equation}
w\left(A_{r}^{\circ}\left(f\chi_{\{f\geq\frac{1}{2}k^{r}\}}\right)\neq0\right)  \le w\left(\bigcup_{y\in\{f\geq\frac{1}{2}k^{r}\}}S(y,r)\right)\leq\sum_{y\in\{f\geq\frac{1}{2}k^{r}\}}w(S(y,r)).
\label{eq:TerminoFacil}
\end{equation}
Observe that choosing
\[
E=\{y\}\qquad F=S(y,s)
\]
in the hypothesis
\[
w(S(y,s))\leq k^{\frac{\beta}{\alpha}ps}w(y)
\]
and consequently
\begin{align*}
w\left(A_{r}^{\circ}\left(f\chi_{\{f\geq\frac{1}{2}k^{r}\}}\right)\neq0\right) & \leq\sum_{y\in\{f\geq\frac{1}{2}k^{r}\}}w(S(y,r))\\
 & \leq k^{\frac{\beta}{\alpha}pr}\sum_{y\in\{f\geq\frac{1}{2}k^{r}\}}w(y)\\
 & =k^{\frac{\beta}{\alpha}pr}w\left(\{f\geq\frac{1}{2}k^{r}\}\right).
\end{align*}
Thus we have that combining the estimates above
\[
\begin{split} & w\left(A_{r}^{\circ}f\geq1\right)\\
 & =w\left(\sum_{\substack{n\in\N\cup\{0\}\\
1\leq2^{n}\leq k^{r}
}
}2^{n}A_{r}^{\circ}\left(\chi_{E_{n}}\right)\geq\frac{1}{2}\right)+w\left(A_{r}^{\circ}\left(f\chi_{\{f\geq\frac{1}{2}k^{r}\}}\right)\neq0\right)\\
 & \leq w\left(\sum_{\substack{n\in\N\cup\{0\}\\
1\leq2^{n}\leq k^{r}
}
}2^{n}A_{r}^{\circ}\left(\chi_{E_{n}}\right)\geq\frac{1}{2}\right)+k^{\frac{\beta}{\alpha}pr}w\left(\{f\geq\frac{1}{2}k^{r}\}\right).
\end{split}
\]
Let $\gamma>0$ to be chosen. Note that if 
\[
\sum_{\substack{n\in\N\cup\{0\}\\
1\leq2^{n}\leq k^{r}
}
}2^{n}A_{r}^{\circ}\left(\chi_{E_{n}}\right)\geq\frac{1}{2}
\]
then we necessarily have for some $n\in\N$, such that $1\leq2^{n}\leq k^{r}$,
\[
A_{r}^{\circ}\left(\chi_{E_{n}}\right)\geq\frac{1}{2^{n+4}}\left(\frac{2^{n}}{k^{r}}\right)^{\gamma}8\left(2^{\gamma}-1\right).
\]
Indeed, otherwise we have that
\[
\begin{split}\frac{1}{2} & \leq8\left(2^{\gamma}-1\right)\sum_{\substack{n\in\N\cup\{0\}\\
1\leq2^{n}\leq k^{r}
}
}2^{n}A_{r}^{\circ}\left(\chi_{E_{n}}\right)\le\frac{8\left(2^{\gamma}-1\right)}{16k^{r\gamma}}\sum_{\substack{n\in\N\cup\{0\}\\
1\leq2^{n}\leq k^{r}
}
}2^{\gamma n}\\
 & \leq8\left(2^{\gamma}-1\right)\frac{2^{\gamma\log_{2}k^{r}}-1}{16k^{r\gamma}\left(2^{\beta}-1\right)}\leq8\left(2^{\beta}-1\right)\frac{k^{r\gamma}-1}{16k^{r\gamma}\left(2^{\gamma}-1\right)}\\
 & <\frac{8\left(2^{\gamma}-1\right)}{16\left(2^{\gamma}-1\right)}=\frac{1}{2}
\end{split}
\]
which is a contraction. Thus 
\[
w\left(A_{r}^{\circ}f\geq2\right)\leq\sum_{\substack{n\in\N\cup\{0\}\\
1\leq2^{n}\leq k^{r}
}
}w(F_{n})+k^{r}Mw\left(f\geq1\right)
\]
where 
\[
F_{n}=\left\{ A_{r}^{\circ}\left(\chi_{E_{n}}\right)\geq\frac{1}{2^{n+4}}\left(\frac{2^{n}}{k^{r}}\right)^{\gamma}8\left(2^{\gamma}-1\right)\right\} .
\]
Note that $F_{n}$ is finite and observe that since $A_{r}^{\circ}$
is a selfadjoint operator,
\[
\begin{split} & \frac{1}{k^{r}}\1\otimes w\left(\{(x,y)\in E_{n}\times F_{n}:d(x,y)=r\}\right)\\
 & =\frac{1}{k^{r}}\sum_{x\in E_{n}}\sum_{\stackrel{y\in F_{n}}{d(x,y)=r}}w(y)\simeq\int_{T}\chi_{E_{n}}A_{r}^{\circ}(w\chi_{F_{n}})(y)=\int_{F_{n}}wA_{r}^{\circ}(\chi_{E_{n}})(y)\\
 & \geq w(F_{n})\frac{1}{2^{n+4}}\left(\frac{2^{n}}{k^{r}}\right)^{\gamma}.
\end{split}
\]
Now, using the hypothesis
\[
\frac{1}{k^{r}}\1\otimes w\left(\{(x,y)\in E_{n}\times F_{n}:d(x,y)=r\}\right)\lesssim k^{r(\beta-1)}w(E_{n})^{\frac{\alpha}{p}}w(F_{n})^{1-\frac{\alpha}{p}}.
\]
Hence
\[
\begin{split} & w(F_{n})\frac{1}{2^{n+4}}\left(\frac{2^{n}}{k^{r}}\right)^{\gamma}\lesssim\frac{1}{8\left(2^{\gamma}-1\right)}k^{r(1-\beta)}w(E_{n})^{\frac{\alpha}{p}}w(F_{n})^{1-\frac{\alpha}{p}}\\
\iff & w(F_{n})^{\frac{\alpha}{p}}\lesssim\frac{1}{2^{\gamma}-1}k^{-r(1-\beta-\gamma)}2^{n-\gamma n}w(E_{n})^{\frac{\alpha}{p}}\\
\iff & w(F_{n})\lesssim\left(\frac{1}{2^{\gamma}-1}\right)^{\frac{p}{\alpha}}k^{-r(1-\beta-\gamma)\frac{p}{\alpha}}2^{(1-\gamma)n\frac{p}{\alpha}}w(E_{n})
\end{split}
\]
Choosing $\gamma=1-t\beta>0$ with $t>1$ we have that
\[
1-\beta-\gamma  =1-\beta-(1-t\beta)=1-\beta-1+t\beta=\beta(t-1)
\]
and 
\[
1-\gamma =1-(1-t\beta)=\beta t=\beta(t-1)+\beta.
\]
Hence
\[
w(F_{n})\lesssim\left(\frac{1}{2^{1-t\beta}-1}\right)^{\frac{p}{\alpha}}\left(\frac{2^{n}}{k^{r}}\right)^{(t-1)\beta\frac{p}{\alpha}}2^{\beta\frac{p}{\alpha}n}w(E_{n})
\]
and we may choose $t=\frac{1+\frac{1}{\beta}}{2}$. Then 
\begin{align*}
w(F_{n}) & \lesssim\left(\frac{2^{n}}{k^{r}}\right)^{\frac{1-\beta}{2\beta}\beta\frac{p}{\alpha}}2^{\beta\frac{p}{\alpha}n}w(E_{n})\\
 & =\left(\frac{2^{n}}{k^{r}}\right)^{\frac{1-\beta}{2}\frac{p}{\alpha}}2^{\beta\frac{p}{\alpha}n}w(E_{n})
\end{align*}
and this yields the desired conclusion.
\end{proof}
Combining the ingredients above we are in the position to settle Theorem
\ref{thm:Suff}. 
\begin{proof}[Proof of Theorem \ref{thm:Suff}]
We begin settling (\ref{eq:Weakbpalpha}). Since as we noted in the
introduction $M^{\circ}f\simeq Mf$ it suffices to settle the result
for $M^{\circ}f$. Since $M^{\circ}f=\sup_{r\geq0}A_{r}^{\circ}f$,
Lemma \ref{lem:sumLevels} implies that 
\begin{align*}
w\left(M^{\circ}f\geq\lambda\right) & \leq\sum_{r=0}^{\infty}w\left(A_{r}^{\circ}f\geq\lambda\right)\\
 & \lesssim\sum_{r=0}^{\infty}\sum_{\substack{n\in\N\cup\{0\}\\
1\leq2^{n}\leq2k^{r}
}
}\left(\frac{2^{n}}{k^{r}}\right)^{\frac{1-\beta}{2}\frac{p}{\alpha}}2^{\beta\frac{p}{\alpha}n}w\left(\left\{ |f|\geq2^{n-1}\lambda\right\} \right)\\
 & \lesssim\sum_{x\in T}\sum_{n=0}^{\infty}\left(\sum_{\substack{r\in\N\cup\{0\}\\
k^{r}\geq2^{n-1}
}
}\frac{1}{k^{r\frac{1-\beta}{2}\frac{p}{\alpha}}}\right)2^{\frac{1-\beta}{2}\frac{p}{\alpha}+\beta\frac{p}{\alpha}n}\chi_{\{|f(x)|\geq2^{n-1}\lambda\}}(x)w(x)\\
 & \lesssim c_{s}\sum_{x\in T}\sum_{n=0}^{\infty}2^{\beta\frac{p}{\alpha}n}\chi_{\{|f(x)|\geq2^{n-1}\lambda\}}w(x)\lesssim c_{s}\sum_{x\in T}\frac{1}{\lambda^{\beta\frac{p}{\alpha}}}|f(x)|^{\beta\frac{p}{\alpha}}w(x).
\end{align*}
Now we turn our attention to the case $\alpha>\beta$. We claim that
\begin{equation}
\sum_{r=0}^{\infty}\|A_{r}^{\circ}f\|_{L^{p}(w)}\lesssim\|f\|_{L^{p}(w)}.\label{eq:SumArBounded}
\end{equation}
Note that from this estimate we can derive both \eqref{eq:Strongpp} and \eqref{eq:StrongDual}.
In the case of \eqref{eq:Strongpp} we have that
\[
\|Mf\|_{L^{p}(w)}\leq\sum_{r=0}^{\infty}\|A_{r}^{\circ}f\|_{L^{p}(w)}\lesssim\|f\|_{L^{p}(w)}.
\]
For \eqref{eq:StrongDual}, we may assume that $f\geq0$. Note that
since
\[
\|Mf\|_{L^{p'}(\sigma)}=\sup_{\|g\|_{L^{p}(w)}=1}\left|\int_{T}Mfg\right|
\]
we can argue as follows 
\begin{align*}
\left|\int_{T}Mfg\right| & \leq\int_{T}Mf|g|\leq\sum_{r=0}^{\infty}\int_{T}A_{r}^{\circ}(f)|g|=\sum_{r=0}^{\infty}\int_{T}fA_{r}^{\circ}(g)\\
 & \leq\|f\|_{L^{p'}(\sigma)}\sum_{r=0}^{\infty}\|A_{r}^{\circ}(g)\|_{L^{p}(w)}\\
 & \lesssim\|f\|_{L^{p'}(\sigma)}\|g\|_{L^{p}(w)}
\end{align*}
and then we are done. Hence we are left with settling \eqref{eq:SumArBounded}.
Note that
\begin{align*}
\|A_{r}^{\circ}f\|_{L^{p}(w)}^{p} & =p\int_{0}^{\infty}\lambda^{p-1}w\left(A_{r}^{\circ}f\geq\lambda\right)d\lambda\\
 & \leq p\int_{0}^{\infty}\lambda^{p-1}w\left(\left\{ |f|\geq2^{n-1}\lambda\right\} \right)d\lambda\\
 & =\sum_{\substack{n\in\N\cup\{0\}\\
1\leq2^{n}\leq2k^{r}
}
}\left(\frac{2^{n}}{k^{r}}\right)^{\frac{1-\beta}{2}\frac{p}{\alpha}}2^{\beta\frac{p}{\alpha}n}p\int_{0}^{\infty}\lambda^{p-1}w\left(\left\{ |f|\geq2^{n-1}\lambda\right\} \right)d\lambda\\
 & =\sum_{\substack{n\in\N\cup\{0\}\\
1\leq2^{n}\leq2k^{r}
}
}\left(\frac{2^{n}}{k^{r}}\right)^{\frac{1-\beta}{2}\frac{p}{\alpha}}2^{\beta\frac{p}{\alpha}n}p\int_{0}^{\infty}\left(\frac{s}{2^{n-1}}\right)^{p-1}w\left(\left\{ |f|\geq s\right\} \right)\frac{1}{2^{n-1}}ds\\
 & =\sum_{\substack{n\in\N\cup\{0\}\\
1\leq2^{n}\leq2k^{r}
}
}\left(\frac{2^{n}}{k^{r}}\right)^{\frac{1-\beta}{2}\frac{p}{\alpha}}2^{\beta\frac{p}{\alpha}n}p\left(\frac{1}{2^{n-1}}\right)^{p}\int_{0}^{\infty}s^{p-1}w\left(\left\{ |f|\geq s\right\} \right)ds\\
 & =\sum_{\substack{n\in\N\cup\{0\}\\
1\leq2^{n}\leq2k^{r}
}
}\left(\frac{2^{n}}{k^{r}}\right)^{\frac{1-\beta}{2}\frac{p}{\alpha}}2^{\beta\frac{p}{\alpha}n}\left(\frac{1}{2^{n-1}}\right)^{p}\|f\|_{L^{p}(w)}^{p}.
\end{align*}
Hence
\[
\|A_{r}^{\circ}f\|_{L^{p}(w)}\leq\left(\sum_{\substack{n\in\N\cup\{0\}\\
1\leq2^{n}\leq2k^{r}
}
}\left(\frac{2^{n}}{k^{r}}\right)^{\frac{1-\beta}{2}\frac{p}{\alpha}}2^{\frac{\beta}{\alpha}np}\left(\frac{1}{2^{n-1}}\right)^{p}\right)^{\frac{1}{p}}\|f\|_{L^{p}(w)}.
\]
Arguing essentially as above we have that
\begin{align*}
\sum_{r=0}^{\infty}\left(\sum_{\substack{n\in\N\cup\{0\}\\
1\leq2^{n}\leq2k^{r}
}
}\left(\frac{2^{n}}{k^{r}}\right)^{\frac{1-\beta}{2}\frac{p}{\alpha}}2^{\frac{\beta}{\alpha}np}\left(\frac{1}{2^{n-1}}\right)^{p}\right)^{\frac{1}{p}} & \leq\sum_{r=0}^{\infty}\sum_{\substack{n\in\N\cup\{0\}\\
1\leq2^{n}\leq2k^{r}
}
}\left(\frac{2^{n}}{k^{r}}\right)^{\frac{1-\beta}{2}\frac{1}{\alpha}}2^{(\frac{\beta}{\alpha}-1)n}\\
 & \lesssim\sum_{n=0}^{\infty}\sum_{\substack{r\in\N\cup\{0\}\\
k^{r}\geq2^{n-1}
}
}\frac{1}{k^{r\frac{1-\beta}{2}\frac{1}{\alpha}}}2^{n}{}^{\frac{1-\beta}{2}\frac{1}{\alpha}}2^{(\frac{\beta}{\alpha}-1)n}\\
 & \lesssim\sum_{n=0}^{\infty}2^{(\frac{\beta}{\alpha}-1)n}<\infty.
\end{align*}
This ends the proof of the Theorem.
\end{proof}
\subsection{Proof of Corollary \ref{cor:SuffAux}}
We begin settling the following Lemma.
\begin{lem}\label{lem:CorSuffAux} Under the conditions of Corollary \ref{cor:SuffAux}
we have that for every pair $E,F$ of finite subsets of $T$ and every
non negative integer $r$,
\[
\1\otimes w\left(\{(x,y)\in E\times F:\ d(x,y)=r\}\right)\le c_{p,\delta}k^{\frac{p}{p-\delta+1}r}w(F)^{1-\frac{1}{p-\delta+1}}w(E)^{\frac{1}{p-\delta+1}}.
\]
\end{lem}
\begin{proof}
We recall that we can split the tree $T$ as $T=\bigcup_{j=0}^{\infty}T_{j}$, where $T_{j}$
is the generation of the tree at depth $j$. We define $E_{j}=E\cap T_{j}$ and $F_{j}=F\cap T_{j}$.
An element in $E_{j}$ and an element in $F_{i}$ can be at distance
exactly $r$, if and only if $i=j+r-2m$ for some $m\in\{0,\ldots,r\}$.
Hence we can write 
\begin{equation}
\begin{split} & \1\otimes w\left(\{(x,y)\in E\times F:d(x,y)=r\}\right)\\
 & =\sum_{m=0}^{r}\sum_{\substack{i,j\in\N\cup\{0\}\\
i=j+r-2m
}
}\1\otimes w\left(\{(x,y)\in E_{j}\times F_{i}:d(x,y)=r\}\right).
\end{split}
\label{eq:levels}
\end{equation}
Now we fix $m\in\{0,\ldots,r\}$ and $i,j\in\N\cup\{0\}$ such that
$i=j+r-2m$. Note that if $x\in T_{j}$ and $y\in T_{i}$ are at distance
$r$ in $T$, then the $m^{th}$ parent of $x$ coincides with the
$(r-m)^{th}$ parent of $y$. This leads to the fact that for each
$y\in T_{i}$ there exist at most $k^{m}$ elements of $x\in T_{j}$
with $d(x,y)=r$. From this it readily follows that
\[
\1\otimes w\left(\{(x,y)\in E_{j}\times F_{i}:d(x,y)=r\}\right)\leq k^{m}w(F_{i}).
\]
On the other hand, by assumption, since $\frac{r+i-j}{2}=r-m$,
\[
\begin{split} & \1\otimes w\left(\{(x,y)\in E_{j}\times F_{i}:d(x,y)=r\}\right)\\
 & =\sum_{x\in E_{j}}w(F_{i}\cap S(x,r))\lesssim\sum_{x\in E_{j}}k^{(r-m)(p-\delta)}k^{r\delta}w(x)\\
 & =k^{(r-m)(p-\delta)}k^{r\delta}w(E_{j}).
\end{split}
\]
Thus combining the ideas above
\begin{equation}
\1\otimes w\left(\{(x,y)\in E_{j}\times F_{i}:d(x,y)=r\}\right)\lesssim\min\left\{ k^{(r-m)(p-\delta)}k^{r\delta}w(E_{j}),k^{m}w(F_{i})\right\} .\label{eq:Levelm}
\end{equation}
Taking into account (\ref{eq:levels}) and (\ref{eq:Levelm}), to
end the proof it suffices to show that
\begin{equation}
\sum_{m=0}^{r}\sum_{\substack{i,j\in\N\cup\{0\}\\
i=j+r-2m
}
}\min\left\{ k^{(r-m)(p-\delta)}k^{r\delta}w(E_{j}),k^{m}w(F_{i})\right\} \le c_{p,\delta}k^{\frac{p}{p-\delta+1}r}w(F)^{1-\frac{1}{p-\delta+1}}w(E)^{\frac{1}{p-\delta+1}}.\label{eq:target}
\end{equation}
Let us define $c_{j}=\frac{w(E_{j})}{k^{(p-\delta)j}}$ and $d_{j}=\frac{w(F_{j})}{k^{j}}$
for $j\geq0$ and $c_{j}=d_{j}=0$ for $j<0$ then, 
\begin{equation}
\sum_{j=0}^{\infty}k^{(p-\delta)j}c_{j}=w(E)\quad\mathrm{and}\quad\sum_{j=0}^{\infty}k^{j}d_{j}=w(F),\label{eq:redfcjdj}
\end{equation}
and we have that
\[
\begin{split}\sum_{m=0}^{r}\sum_{\substack{i,j\in\N\cup\{0\}\\
i=j+r-2m
}
}\left\{ k^{(r-m)(p-\delta)}k^{r\delta}w(E_{j}),k^{m}w(F_{i})\right\}  & =\sum_{m=0}^{r}\sum_{\substack{i,j\in\N\cup\{0\}\\
i=j+r-2m
}
}\min\left\{ k^{(p-\delta)(r-m+j)}k^{\delta r}c_{j},k^{m}k^{i}d_{i}\right\} \\
 & =\sum_{m=0}^{r}\sum_{\substack{i,j\in\N\cup\{0\}\\
i=j+r-2m
}
}\min\left\{ k^{\delta r}k^{\frac{(i+j+r)(p-\delta)}{2}}c_{j},k^{\frac{i+j+r}{2}}d_{i}\right\}.
\end{split}
\]
Taking the identity above into account, settling (\ref{eq:target})
reduces to show that 
\[
\sum_{m=0}^{r}\sum_{\substack{i,j\in\N\cup\{0\}\\
i=j+r-2m
}
}\min\left\{ k^{\delta r}k^{\frac{(i+j+r)(p-\delta)}{2}}c_{j},k^{\frac{i+j+r}{2}}d_{i}\right\} \leq c_{p,\delta}k^{\frac{p}{p-\delta+1}r}w(F)^{1-\frac{1}{p-\delta+1}}w(E)^{\frac{1}{p-\delta+1}}.
\]
To prove this inequality, we let $\rho$ be a real parameter to be
chosen later, and argue as follows 
\begin{eqnarray*}
 &  & \sum_{m=0}^{r}\sum_{\substack{i,j\in\N\cup\{0\}\\
i=j+r-2m
}
}\min\left\{ k^{\delta r}k^{\frac{(i+j+r)(p-\delta)}{2}}c_{j},k^{\frac{i+j+r}{2}}d_{i}\right\} \\
 &  & \leq k^{\frac{p+\delta}{2}r}\sum_{\substack{i,j\in\N\cup\{0\}\\
i<j+\rho
}
}k^{\frac{(i+j)(p-\delta)}{2}}c_{j}+k^{\frac{r}{2}}\sum_{\substack{i,j\in\N\cup\{0\}\\
i\geq j+\rho
}
}k^{\frac{i+j}{2}}d_{i}\\
 &  & \leq k^{\frac{p+\delta}{2}r}\sum_{j=0}^{\infty}k^{\frac{(j+\rho+j)(p-\delta)}{2}}c_{j}+k^{\frac{r}{2}}\sum_{i=0}^{\infty}k^{i-\frac{\rho}{2}}d_{i}\\
 &  & =k^{\frac{p+\delta}{2}r}k^{\frac{\rho(p-\delta)}{2}}\sum_{j=0}^{\infty}k^{j(p-\delta)}c_{j}+k^{\frac{r}{2}}k^{-\frac{\rho}{2}}\sum_{i=0}^{\infty}k^{i}d_{i}\\
 &  & =k^{\frac{p+\delta}{2}r}k^{\frac{\rho(p-\delta)}{2}}w(E)+k^{\frac{r}{2}}k^{-\frac{\rho}{2}}w(F).
\end{eqnarray*}
From this point the desired result follows optimizing on $\rho$. For reader's convenience we provide our optimization argument. Let \[f_{\delta,p,r}(\rho)=k^{\frac{p+\delta}{2}r}k^{\frac{\rho(p-\delta)}{2}r}w(E)+k^{\frac{r}{2}}k^{-\frac{\rho}{2}}w(F).\]
First note that
\begin{align*}
f'_{\delta,p,r}(\rho) & =\frac{p-\delta}{2}k^{\frac{p+\delta}{2}r}k^{\frac{\rho(p-\delta)}{2}}\ln(k)w(E)-\frac{1}{2}\ln(k)k^{\frac{r}{2}}k^{-\frac{\rho}{2}}w(F),\\
f''_{\delta,p,r}(\rho) & =\left(\frac{p-\delta}{2}\right)^{2}k^{\frac{(\delta-\alpha)r}{2}}k^{\frac{\rho(p-\delta)}{2}}\ln^{2}(k)w(E)+\frac{1}{4}\ln^{2}(k)k^{\frac{r}{2}}k^{-\frac{\rho}{2}}w(F)>0
\end{align*}
and hence at the critical point the minimum is attained. Now we compute
that critical point.
\begin{align*}
 & \frac{p-\delta}{2}k^{\frac{p+\delta}{2}r}k^{\frac{\rho(p-\delta)}{2}}\ln(k)w(E)=\frac{1}{2}\ln(k)k^{\frac{r}{2}}k^{-\frac{\rho}{2}}w(F)\\
\iff & (p-\delta)k^{\frac{p+\delta}{2}r+\frac{\rho(p-\delta)}{2}-\frac{r}{2}+\frac{\rho}{2}}w(E)=w(F)\\
\iff & k^{\frac{p+\delta}{2}r+\frac{\rho(p-\delta)}{2}-\frac{r}{2}+\frac{\rho}{2}}=\frac{w(F)}{w(E)(p-\delta)}\\
\iff & (p+\delta-1)r+\rho(p-\delta+1)=2\log_{k}\left(\frac{w(F)}{w(E)(p-\delta)}\right)\\
\iff & \rho=\frac{2\log_{k}\left(\frac{w(F)}{w(E)(p-\delta)}\right)}{p-\delta+1}-\frac{(p+\delta-1)r}{p-\delta+1}.
\end{align*}
Hence $f_{\delta,p,r}(\rho)$ reaches its absolute minimum at $\rho=\frac{2\log_{k}\left(\frac{w(F)}{w(E)(p-\delta)}\right)}{p-\delta+1}-\frac{(p+\delta-1)r}{p-\delta+1}.$
Finally we observe that
\begin{align*}
f_{\delta,p,r}(\rho) & =k^{\frac{p+\delta}{2}r}k^{\left[\frac{2\log_{k}\left(\frac{w(F)}{w(E)(p-\delta)}\right)}{p-\delta+1}-\frac{(p+\delta-1)r}{p-\delta+1}\right]\frac{(p-\delta)}{2}}w(E)+k^{\frac{r}{2}}k^{-\frac{1}{2}\left[\frac{2\log_{k}\left(\frac{w(F)}{w(E)(p-\delta)}\right)}{p-\delta+1}-\frac{(p+\delta-1)r}{p-\delta+1}\right]}w(F)\\
 & =k^{\left(\frac{p+\delta}{2}-\frac{(p+\delta-1)}{p-\delta+1}\frac{(p-\delta)}{2}\right)r}k^{\left[\log_{k}\left(\frac{w(F)}{w(E)(p-\delta)}\right)\right]\frac{p-\delta}{p-\delta+1}}w(E)+k^{\frac{r}{2}\left(1+\frac{(p+\delta-1)}{p-\delta+1}\right)}k^{-\left[\frac{\log_{k}\left(\frac{w(F)}{w(E)(p-\delta)}\right)}{p-\delta+1}\right]}w(F)\\
 & =k^{\left(\frac{p+\delta}{2}-\frac{(p+\delta-1)}{p-\delta+1}\frac{(p-\delta)}{2}\right)r}\left(\frac{w(F)}{w(E)(p-\delta)}\right)^{\frac{p-\delta}{p-\delta+1}}w(E)+k^{\frac{p}{p-\delta+1}r}\left(\frac{w(F)}{w(E)(p-\delta)}\right)^{-\frac{1}{p-\delta+1}}w(F)\\
 & \leq c_{p,\delta}k^{\frac{p}{p-\delta+1}r}w(F)^{\frac{p-\delta}{p-\delta+1}}w(E)^{1-\frac{p-\delta}{p-\delta+1}}=c_{p,\delta}k^{\frac{p}{p-\delta+1}r}w(F)^{1-\frac{1}{p-\delta+1}}w(E)^{\frac{1}{p-\delta+1}}
\end{align*}
and this yields the desired conclusion.
\end{proof}
At this point we are in the position to settle Corollary \ref{cor:SuffAux}.
\begin{proof}[Proof of Corollary \ref{cor:SuffAux}]
First observe that for the weak type estimate, by Lemma \ref{lem:CorSuffAux}, \eqref{eq:SuffCond} is satisfied choosing $\beta=\alpha=\frac{p}{p-1+\delta}$. For the strong type estimates note that also due to Lemma \ref{lem:CorSuffAux},  \eqref{eq:SuffCond} is satisfied choosing $\beta=\frac{p}{p-1+\delta}$ and $\alpha=\frac{q}{p-1+\delta}$, which in turn yields the desired conclusion.
\end{proof}

\section{Two weight estimates}\label{sec:twoWeights}
In this last section we gather our results regarding two weight estimates. The first of them is that with minor adjustments it is possible to prove the following two weighted version of Theorem \ref{thm:Suff}.
\begin{thm}
Let $k\geq2$ be an integer. Let $p>1$ and let $u,v$
be weights. Assume that there exist $0<\beta<1$ and $\beta\leq\alpha<p$
such that 
\begin{equation*}
\1\otimes u\left(\{(x,y)\in E\times F:\ d(x,y)=r\}\right)\lesssim k^{r\beta}v(E)^{\frac{\alpha}{p}}u(F)^{1-\frac{\alpha}{p}}.\label{eq:SuffCond2w}
\end{equation*}
Then 
\begin{equation*}
\|Mf\|_{L^{\frac{\beta}{\alpha}p,\infty}(u)}\lesssim\|f\|_{L^{\frac{\beta}{\alpha}p}(v)}\label{eq:Weakbpalpha2w}
\end{equation*}
Furthermore, if $\beta<\alpha$ then 
\begin{equation*}
\|Mf\|_{L^{p}(u)}\lesssim\|f\|_{L^{p}(v)}\label{eq:Strongpp2w}
\end{equation*}
and also 
\begin{equation*}
\|Mf\|_{L^{p'}(\sigma_{u,p})}\lesssim\|f\|_{L^{p'}(\sigma_{v,p})}\label{eq:StrongDual2w}
\end{equation*}
where $\sigma_{\rho,p}=\rho^{-\frac{1}{p-1}}$.
\end{thm}
From this Theorem it is also possible to derive the following Corollary.
\begin{cor}
Let $u,v$ be a weight such that there exists $0<\delta<1$
such that if $x\in T_{j}$, $F\subset T_{i}$ and $|i-j|\leq r$,
\begin{equation*}
u(F\cap S(x,r))\lesssim k^{\frac{i-j+r}{2}(p-\delta)}k^{r\delta}v(x).
\end{equation*}
Then 
\begin{equation*}
\|Mf\|_{L^{p}(u)}\lesssim\|f\|_{L^{p}(v)}
\end{equation*}
Furthermore, if $p<q$ then 
\begin{equation*}
\|Mf\|_{L^{q}(u)}\lesssim\|f\|_{L^{q}(v)}
\end{equation*}
and also 
\begin{equation*}
\|Mf\|_{L^{p'}(\sigma_{u,p})}\lesssim\|f\|_{L^{p'}(\sigma_{v,p})}
\end{equation*}
where $\sigma_{\rho,p}=\rho^{-\frac{1}{p-1}}$.
\end{cor}
At this point a natural question would be to consider whether testing conditions like the ones introduced by Sawyer in \cite{STesting} are sufficient for strong type estimates to hold. Our next result shows that even in the one weight setting that is not the case.
\begin{thm} Let $p>1$ and
\[
w(x)=\sum_{j=0}^{\infty}k^{(p-1)j}\chi_{T_{j}}(x).
\]
Then we have that 
\[
\|Mf\|_{L^{p}(w)}\lesssim\|f\|_{L^{p}(w)}
\]
does not hold but for every ball $B$
\[
\int_{B}M(\chi_{B}\sigma)^{p}w\lesssim\int_{B}\sigma
\]
where $\sigma=w^{-\frac{1}{p-1}}$.
 \end{thm}

\begin{proof} By Proposition \ref{Prop:Neg2} we know that the weighted strong type estimate does not hold for $w$. 
Now we observe that note that 
\[
\sigma(x)=\sum_{j=0}^{\infty}\frac{1}{k^{j}}\chi_{T_{j}}(x).
\]
For this weight it was shown in \cite{ORRS} that 
\[
M\sigma(x)\lesssim\sigma(x).
\]
Taking this into account 
\[
\int_{B}M(\chi_{B}\sigma)^{p}w\lesssim\int_{B}\sigma{}^{p}w=\int_{B}\sigma
\]
and the testing condition holds. \end{proof}

\bibliographystyle{plain}
\bibliography{refs}

\end{document}